\documentclass[a4paper,12pt]{amsart}

\usepackage[ngerman, italian, english]{babel}
\usepackage{amssymb}
\usepackage{amsmath}
\usepackage{hyperref}
\usepackage{mathptmx}
\usepackage{verbatim}
\usepackage[utf8]{inputenc}

\theoremstyle{plain}
\newtheorem{thm}{\bf Theorem}[section]

\newtheorem{prop}[thm]{\bf Proposition}
\newtheorem{lemma}[thm]{\bf Lemma}
\newtheorem{cor}[thm]{\bf Corollary}

\theoremstyle{definition}

\theoremstyle{remark}
\newtheorem{remark}[thm]{\bf Remark}
\newtheorem{example}[thm]{\bf Example}
\numberwithin{equation}{section}

\def\rk{\operatorname{rank}}

\def \chara{{\operatorname{char}}}

\def\sat{\operatorname{sat}}
\def \reg{{\operatorname{reg}}}

\def\reg{\operatorname{reg}}

\def\GL{\operatorname{GL}}

\def \mm{{\mathfrak{m}}}

\def \PP{\mathbb P}

\def \RR{\mathbb R}

\def \G{\mathcal G}

\def \R{\mathcal R}

\newcommand{\HH}[3]{\operatorname{H}^{#1}_{#2}(#3)}

\def \P{\mathcal P}

\def \F{\mathcal F}
\def \e{\varepsilon}

\def \D{\Delta}

\def \NN{{\mathbb{N}}}

\def \ZZ{{\mathbb{Z}}}

\def \Quot{{\operatorname{Quot}}}

\def\ls{\leqslant}
\def\gs{\geqslant}
\def\vol{\operatorname{vol}}
\def\diam{\operatorname{diam}}
\def\det{\operatorname{det}}
\def\tr{\operatorname{tr}}
\def\length{\lambda}
\def\iff{\Longleftrightarrow}
\def\fm{\mathfrak{m}}

\newcommand{\ddd}[1]{\:\mathrm{d}{#1}}
\newcommand{\mcal}{\mathcal}
\newcommand{\mdd}[1]{\,(\mathrm{mod}\,\,\, #1)}

\title[Multiplicities of Classical Varieties]{Multiplicities of Classical Varieties}
\author[Jack Jeffries]{Jack Jeffries}
\address{Jack Jeffries \\ Department of Mathematics \\ University of Utah \\ 155 S 1400 E\\
Salt Lake City, UT. 84112}
\email{jeffries@math.utah.edu}
\author[Jonathan Monta\~{n}o]{Jonathan Monta\~{n}o}
\address{Jonathan Monta\~{n}o \\ Department of Mathematics \\ Purdue University \\150 North University Street \\West Lafayette, IN. 47907}
\email{jmontano@purdue.edu}
\author[Matteo Varbaro]{Matteo Varbaro}
\address{Matteo Varbaro \\ Dipartimento di Matematica \\ Universit\`a di Genova \\Via Dodecaneso, 35 \\16146 Genova - ITALY}
\email{varbaro@dima.unige.it}

\begin{document}

\thanks{This material is based upon work supported by the National Science Foundation under Grant No~0932078 000, while the authors were in residence at the Mathematical Science Research Institute in Berkeley, California, during the Fall semester of 2012. The first author was also supported in part by NSF grants DMS~0758474 and DMS~1162585, the second author by NSF grant DMS~0901613, the third author by the Postdoctoral Fellowship offered by the MSRI for the Commutative Algebra program.}

\begin{abstract}
The $j$-multiplicity plays an important role in the intersection theory of St\"uckrad-Vogel cycles, while recent developments confirm the connections between the $\epsilon$-multiplicity and equisingularity theory. In this paper we establish, under some constraints, a relationship between the $j$-multiplicity of an ideal and the degree of its fiber cone. As a consequence, we are able to compute the $j$-multiplicity of all the ideals defining rational normal scrolls. By using the standard monomial theory, we can also compute the $j$- and $\epsilon$-multiplicity of ideals defining determinantal varieties: The found quantities are integrals which, quite surprisingly, are central in random matrix theory.
\end{abstract}

\maketitle

\section{Introduction}

For a long time, the development of the theory of multiplicities in local rings has been supplying essential techniques for the study of local algebra, intersection theory, and singularity theory in algebraic geometry. The Hilbert-Samuel multiplicity is defined for {\it $\mm$-primary} ideals $I$ in a noetherian local or graded ring $(R,\mm)$ as
\begin{align*}
e(I)&=\lim_{s\to \infty}\frac{(d-1)!}{s^{d-1}}\,\lambda({I^s/I^{s+1}})\\
&=\lim_{s\to \infty}\frac{d!}{s^{d}}\,\lambda({R/I^{s+1}})\,.
\end{align*}
Also, the degree $e(R)$ of $R$ is defined as $e(\mm)$.
Classically, the Hilbert-Samuel multiplicity is used to define the intersection numbers for varieties. Other significant applications include Rees' criterion for integral dependence, and depth conditions for the associated graded algebras (\cite{RV,SA1,SA2}). However, the restriction to $\mm$-primary ideals is a strong limit on the applicability of the techniques provided by the Hilbert-Samuel multiplicity.

This led, in the 1993 paper \cite{AM} of Achilles and Manaresi, to the introduction of the $j$-multiplicity $j(I)$ of an ideal $I$ (not necessarily $\mm$-primary) of a noetherian local ring $(R,\mm)$. This multiplicity provides a local algebraic foundation to the intersection theory of St\"uckrad-Vogel cycles. Using $j$-multiplicity, many key results have had generalizations to all ideals including a numerical criterion for integral dependence \cite{FM}, and depth conditions on the associated graded algebra \cite{MX,PX}. 

However, there is still an evident lack of examples of ideals whose $j$-multiplicity is known. The main tool to compute this invariant is a formula interpreting the $j$-multiplicity as the length of a suitable $R$-module, proved in various forms and levels of generality in \cite{AM,FM1,Xie}. Such a length-formula allowed Nishida and Ulrich to compute, in \cite{NU}, the $j$-multiplicity of the ideals defining the rational normal curve of degree 4 in $\PP^4$ and the rational normal scroll $\PP(2,3)\subseteq \PP^7$. There are two major limitations to this approach: the suitable $R$-module is defined in terms of ``sufficiently general'' elements in $I$, a condition which cannot be verified directly and dissolves any extra (e.g., combinatorial) structure that $I$ might have. Also, determining lengths by computational methods such as computing Gr\"obner bases can be quite slow. An exception is monomial ideals of a polynomial ring (or normal affine toric ring) in $d$ variables. In this case the first two authors of this paper could recently show in \cite{JM} that the $j$-multiplicity is the volume of a suitable polytopal complex in $\RR^d$, which is much more amenable to quick computation than the length formula.

Another related multiplicity for an arbitrary ideal is the  $\epsilon$-multiplicity of an ideal $I\subseteq R$, which was first introduced by Ulrich and Validashti in \cite{UV} as a generalization of the Buchsbaum-Rim multiplicity. This multiplicity has close connections to volume of divisors and has found applications in equisingularity theory \cite{KUV}. This invariant exhibits more mysterious behavior than $j$-multiplicity, and little is known about calculating it.

 The main results of the present paper are:
\begin{itemize}
\item[(i)] Theorem~\ref{thm:j=te} (iii): If $I$ is a homogeneous ideal of a standard graded domain $A$, such that $I$ has maximal analytic spread ($=\dim(A)$), is generated in a single degree $t$ and all its $s^{\text{th}}$ powers are saturated in degrees $\gs st$, then $j(I)$ is equal to the degree of the fiber cone of $I$ multiplied by $t$.
\item[(ii)] Theorem~\ref{j-scrolls}: A closed formula for the $j$-multiplicity of the ideals defining any rational normal scroll $V$, depending only on the dimension and the codimension of $V$.
\item[(iii)] Theorem~\ref{thm_generic}: A computation of the $j$-multiplicity of $I_t(X)$, the $\epsilon$-multiplicity of $I_t(X)$, and the degree of the fiber cone of $I_t(X)$, where $I_t(X)$ denotes the ideal generated by the $t$-minors of a generic $m\times n$-matrix $X$. To give an idea, $j(I_t(X))$ is the integral of a symmetric polynomial over a hypersimplex in $\RR^{\min\{m,n\}}$. Surprisingly, such a quantity has a precise meaning in random matrix theory. Similar results are provided for minors (resp. pfaffians) of generic symmetric (resp. skew-symmetric) matrices, see Theorems~\ref{thm_symmetric} and~\ref{thm_pfaffian}.
\end{itemize}

\vspace{2mm}

Given an ideal $I$ in a $d$-dimensional noetherian local ring $(R,\mm)$, its {\it $j$-multiplicity} is defined as:
\[j(I)=\lim_{s\to \infty}\frac{(d-1)!}{s^{d-1}}\,\lambda(\HH{0}{\mm}{I^s/I^{s+1}})\,.\]
This number is a nonnegative integer that agrees with the Hilbert-Samuel multiplicity when $I$ is $\mm$-primary (in which case $I^s/I^{s+1}$ is already $\mm$-torsion). However one can show that $j(I)\neq 0$ if and only if the analytic spread of $I$ is maximal, and this is the case in a much larger variety of situations than the $\mm$-primary case. The {\it $\epsilon$-multiplicity} of an ideal $I\subseteq R$ is defined as
\[\epsilon(I)=\lim_{s\to \infty}\frac{d!}{s^{d}}\,\lambda(\HH{0}{\mm}{R/I^{s+1}})\,.\]
The $\epsilon$-multiplicity has recently been shown by Cutkosky \cite{Cut2} to exist for any ideal in an analytically unramified ring; however, there are examples, e.g., \cite{CHST}, in which the $\epsilon$-multiplicity is irrational. Of course, these two multiplicities can be defined for any homogeneous ideal $I$ in a graded ring $R=\bigoplus_{k\gs 0} R_k$, where $(R_0,\mm_0)$ is local and the role of $\mm$ is played by $\mm_0\oplus (\bigoplus_{k>0}R_k)$ (in general we have to define the $\epsilon$-multiplicity as a $\limsup$). 

After a preliminary section recalling the basic facts about these multiplicities, in Section \ref{sec:j=te} we prove Theorem \ref{thm:j=te}. Firstly we show that, if $I$ is an ideal with maximal analytic spread, generated in a single degree $t$, of a standard graded domain $A$, then $j(I)=k\cdot e(\F(I))$ for some integer $k\gs t$, where $\F(I)$ denotes the fiber cone of $I$. In general $k$ may be bigger than $t$, however we prove that, if 
\begin{equation}\label{intro1}\tag{\dag}
[(I^s)^{\sat}]_r=[I^s]_r \mbox{ for all $s\gg 0$ and $r\gs st$},
\end{equation}
then $j(I)=t\cdot e(\F(I))$. Property \eqref{intro1} is rather strong; however, there are interesting classes of ideals satisfying it. For example, by using a result in \cite{BCV}, ideals defining rational normal scrolls satisfy \eqref{intro1}. Furthermore, the degree of the fiber cone of such ideals can be computed combining results in \cite{CHV} and \cite{BCV}, so we are able to compute, in Theorem \ref{j-scrolls}, $j(I)$ for every ideal defining a rational normal scroll.

Other ideals satisfying \eqref{intro1} are ideals generated by the $t$-minors $I_t(X)$ of the $m\times n$ generic matrix $X$. Theorem~\ref{thm:j=te} applies, so one could deduce $j(I_t(X))$ from $e(\F(I_t(X)))$. The problem is that $\F(I_t(X))$ is the algebra of minors, denoted by $A_t$ in \cite{BC}, which is a quite subtle algebra; for example, to describe its defining equations is a very difficult open problem (see \cite{BCV1}). Also, $e(A_t)$ is unknown (problem (b) at the end of the introduction in \cite{BC}), therefore Theorem~\ref{thm:j=te} is not decisive at this time (apart from special cases --- for example if $t=m-1=n-1$ then $A_t$ is a polynomial ring, so $j(I_t(X))=t$ in this case). We note that the analytic spread of such ideals is maximal if and only if $t<\min\{m,n\}$: see Remark~\ref{analytic_spread} for details.

We therefore aim to compute simultaneously the $j$-multiplicity and the $\epsilon$-multiplicity of $I_t(X)$, exploiting the standard monomial theory and the description of the primary decomposition of the (integral closure of the) powers of $I_t(X)$. These facts are gathered in Section~\ref{sec:standardmon}, where the analog statements concerning symmetric and skew-symmetric matrices are also treated. In Section~\ref{sec:count}, by capitalizing on the hook length formula, we express the number of standard tableaux in a fixed alphabet of a given shape (Young diagram) $(a_1,\ldots ,a_p)$ as a polynomial function in $r_i=|\{j:a_j=i\}|$ (see Proposition~\ref{W-prop}). This allows us to express $j(I_t(X))$ as an integral of a polynomial over a hypersimplex (which is a well studied polytope, see for example \cite{Sta2}). Precisely (assuming that $m\ls n$):
\begin{equation}\label{intro2}
j(I_t(X))=ct \!\int\limits_{\substack{[0,1]^m \\ \sum{z_i}= t}}(z_1\cdots z_{m})^{n-m}\prod_{1\ls i < j\ls m}(z_j-z_i)^2\ddd{\nu}\,, 
\end{equation}
where $c=\frac{(nm-1)!}{(n-1)!\cdots (n-m)!m!\cdots 1!}$ (see Theorem~\ref{thm_generic} (i)). Analogous formulas are obtained for the $\epsilon$-multiplicity and for the degree of the algebra of minors (Theorem~\ref{thm_generic} (ii) and (iii)). 

The integral in \eqref{intro2} is surprisingly related to quantities in random matrix theory, as explained in Section \ref{sec:integral}. When ${t=1}$, the exact value of the integral in \eqref{intro2} is known by unpublished work of Selberg (see \cite{FW} for a detailed account). However, since $I_1(X)$ is the irrelevant ideal, we knew a priori that ${j(I_1(X))=1}$.  On the other hand, if ${t>1}$, at the moment a closed formula for the exact value of the integral in \eqref{intro2} is unknown; at least, however, it is possible to get a $N$-variate power series whose coefficients are the values of the integrals over a simplex of the products of $N$ fixed linear forms (see \cite{BBDKV}). The integrand in \eqref{intro2} is a product of linear forms, and the region over which we integrate is a hypersimplex, which has many well-known triangulations (for example see \cite{Sta2}, \cite{Stuu}). We can therefore apply the method of \cite{BBDKV}. When $t=m-1$ our region is already a simplex, so in this case the situation is simpler. These aspects are discussed in Section \ref{sec:integral}. 

The analogs of the above results are also reported for $t$-minors of a generic symmetric matrix and $2t$-pfaffians of a generic skew-symmetric matrix.

\section{Preliminaries}

In this section $R$ will be a noetherian local ring and $\mm$ its unique maximal ideal, or an $\NN$-graded noetherian ring $\bigoplus_{i\gs 0}R_i$ with $(R_0,\mm_0)$ local and ${\mm=\mm_0\oplus(\bigoplus_{i>0}R_i)}$. Given an ideal $I\subseteq R$ (homogeneous if $R$ is graded), we will use throughout the following notation:
\begin{itemize}
\item[(i)] $\R(I)$ $=$ $\bigoplus_{s\gs 0}I^s$, the {\it Rees ring};
\item[(ii)] $\G(I)$ $=$ $\bigoplus_{s\gs 0}I^s/I^{s+1}=\R(I)/I\R(I)$, the {\it associated graded ring};
\item[(iii)] $\F(I)$ $=$  $\bigoplus_{s\gs 0}I^s/\mm I^s=\R(I)/\mm \R(I)=\G(I)/\mm \G(I)$, the {\it fiber cone}.
\end{itemize} 
The $j$-multiplicity of $I$ can also be computed using the normal filtration $\{\overline{I^s}\}_{s\gs 0}$, which we will use to deal with the cases in which $R$ has an exceptional characteristic with respect to the ideal $I$. For a proof of this statement, we will use the notion of $j$-multiplicity of an $R$-module $M$ (in the graded setting).

The $j$-multiplicity of $M$ with respect to $I$ is defined as 
\[j(I,M)=\displaystyle\lim_{s\rightarrow \infty} \frac{(d-1)!}{s^{d-1}}\,\lambda\big(\HH{0}{\mm}{I^sM/I^{s+1}M}\big)\]
 and it is additive on short exact sequences; i.e., if 
\[0\rightarrow M'\rightarrow M\rightarrow M''\rightarrow 0\,,\] 
then $j(I,M)=j(I,M')+j(I,M'')$, see for example \cite[Theorem~3.11]{NU}.

\begin{prop}\label{j-prop}
If $R$ is analytically unramified and $\dim(R)>0$, then
\begin{itemize}
\vspace{.1 cm}
\item[(a)] $\displaystyle j(I)=\displaystyle\lim_{s\rightarrow \infty} \frac{(d-1)!}{s^{d-1}}\,\lambda\Big(\big((\overline{I^{s+1}})^{\sat} \cap \overline{I^s}\,\big) / \overline{I^{s+1}}\,\Big)\,$ and \\ \vspace{.1 cm}
\item[(b)] $\displaystyle \e(I)=\displaystyle\lim_{s\rightarrow \infty} \frac{d!}{s^{d}}\,\lambda\Big((\overline{I^{s+1}})^{\sat}\, / \overline{I^{s+1}}\,\Big)\,.$
\end{itemize}
\vspace{.1 cm}
\end{prop}
\begin{proof} By \cite[Corollary~9.2.1]{SH} there exists an integer $\l\gs 1$ such that ${\overline{I^{n+l}}}$ is equal to $I^n\overline{I^l}$ for every $n\gs 0$.\\
(a)\   Here, the conclusion follows from the additivity of $j$-multiplicity in the short exact sequence 
\[0\rightarrow \overline{I^l}\rightarrow R\rightarrow R/\overline{I^l}\rightarrow 0\,,\]
 and by noting that $j(I,R/\overline{I^l})=0.$ 
 
(b) \ This argument is similar to the one in \cite[Theorem~5.1~part~(b)]{JM}.
We have the following two short exact sequences for every $s\gs l$:
\[0\rightarrow \overline{I^s}/I^s \rightarrow R/I^s \rightarrow R/\overline{I^s} \rightarrow 0 \,,\]
\[0\rightarrow I^{s-l}/\overline{I^{s}} \rightarrow R/\overline{I^{s}} \rightarrow R/I^{s-l} \rightarrow 0\,. \]
Since $\lambda_R\big(\HH{0}{\fm}{-}\big)$ is subadditive on short exact sequences, we obtain the following inequalities
\begin{equation}\label{ineq}
\begin{aligned}
 \lambda_R\big(&\HH{0}{\fm}{R/I^{s}}\big) \ls \lambda_R\big(\HH{0}{\fm}{\overline{I^s}/I^{s}}\big) \!+\! \lambda_R\big(\HH{0}{\fm}{R/\overline{I^s}}\big) \\
&\ls \lambda_R\big(\HH{0}{\fm}{\overline{I^s}/I^{s}}\big) \!+\! \lambda_R\big(\HH{0}{\fm}{R/I^{s-l}}\big)\!+\!\lambda_R\big(\HH{0}{\fm}{I^{s-l}/\overline{I^{s}}}\big).
\end{aligned}
\end{equation}
Observe that 
\[\lambda_R\big(\HH{0}{\fm}{\overline{I^s}/I^s}\big)\ls \lambda_R\big(\HH{0}{\fm}{I^{s-l}/I^s}\big)\ls \sum_{i=0}^{l-1}\lambda_R\big(\HH{0}{\fm}{I^{s-l+i}/I^{s-l+i+1}}\big)\,.\]
It follows that
\[
\limsup_{s\rightarrow \infty}\frac{d!}{s^d}\,\lambda_R\big(\HH{0}{\fm}{\overline{I^s}/I^s}\big)\ls \sum_{i=0}^{l-1}\limsup_{s\rightarrow \infty} \frac{d!}{s^d}\,\lambda_R\big(\HH{0}{\fm}{I^{s-l+i}/I^{s-l+i+1}}\big)\big)=0\,,\]
where the last equality holds by the definition of $j$-multiplicity.
Therefore $\lim _{n\rightarrow \infty}\frac{d!}{s^d}\lambda_R\big(\HH{0}{\fm}{\overline{I^s}/I^s}\big)=0$. Similarly we can conclude using part (a) that $\lim _{n\rightarrow \infty}\frac{d!}{s^d}\lambda_R\big(\HH{0}{\fm}{I^{s-l}/\overline{I^{s}}}\big)=0$.
\noindent Using these two limits in the two inequalities of $\eqref{ineq}$, we obtain
\begin{equation}\label{inf_sup}\begin{aligned}
\liminf_{s\rightarrow \infty}\frac{d!}{s^d}\,\lambda_R\big(\HH{0}{\fm}{R/I^{s}}\big)
&\ls \liminf_{s\rightarrow \infty}\frac{d!}{s^d}\,\lambda_R\big(\HH{0}{\fm}{R/\overline{I^s}}\big)\\
 &\ls \limsup_{s\rightarrow \infty}\frac{d!}{s^d}\,\lambda_R\big(\HH{0}{\fm}{R/\overline{I^s}}\big)\\
 &\ls \limsup_{s\rightarrow \infty}\frac{d!}{s^d}\,\lambda_R\big(\HH{0}{\fm}{R/I^s}\big)
\end{aligned}\end{equation}
Then, by \cite[Corollary~6.3]{Cut2},  the first $\liminf$ and last $\limsup$ in \eqref{inf_sup} agree (that is, the $\e$-multiplicity of $I$ exists as a limit), so equality holds throughout. 
\end{proof}
The following lemma will be crucial to show the results in Section \ref{section_main}.

\begin{lemma}[Integration Lemma]\label{integration} Let $f$ be a polynomial in $\RR[s,x_1,\dots,x_d]$ of degree $e$ and write
\[ f(s,x) = f_{e} (s,x) + f_{e-1} (s,x) + \dots + f_0 (s,x) \]
as a sum of homogeneous forms. Fix a positive integer $k$, a $d\times k$ matrix with real entries $(a_{ij})_{ij}$, a vector $(b_i)_i\in \RR^d$, two vectors $(p_j)_j, (q_j)_j\in\ZZ^k$, and three natural numbers $r<t<m$.
For any $s\in \NN$ let 
\[G(s)=\big\{x\in \ZZ^d \,\big|\, \forall j=1,\ldots ,k,\, \sum a_{ij} x_i \ls p_j s + q_j, \sum b_{i} x_i \equiv t s + r \,\mdd{m}\big\}\,. \]
Then, denoting by $\mcal{P}$ the region in $\RR^d$ given by $\sum a_{ij} x_i \ls p_j$ for $j=1,\ldots ,k$,
\[\lim_{s\rightarrow \infty} \frac{m}{s^{d+e}}\sum_{x\in G(s)} f(s,x_1,\dots,x_d) = \int_\mcal{P} f_{e}(1,x_1,\dots,x_d)\ddd x \,.\]
\end{lemma}
\begin{proof} 
Let $\Delta$ be a fundamental domain for the lattice $L\subseteq \RR^d$ defined as $\{x\in\ZZ^d:\sum b_i x_i \equiv 0 \mdd{m}\}$, and for all $s\in\NN$ put
 \[ \mcal{P}_s = \big\{x\in \RR^d \,|\, \forall j=1,\ldots ,k,\, \sum a_{ij} x_i \ls p_j s + q_j \big\} \,.\]
Pick some arbitrary $\alpha \in G(s)$, so that $G(s)=\mcal{P}_s \cap (L+\alpha)$. Clearly $\frac{1}{s} \Delta$ is a fundamental domain for $\frac{1}{s} L$ with an associated tiling 
\[\frac{1}{s}\Big(L+\Delta\Big) = \Big\{ \frac{1}{s}(l+\Delta)\,|\, l\in L \Big\}\,.\]
 This tiling intersected (element-by-element) with the region $\mcal{P}$ forms a partition of  $\mcal{P}$ such that each nonempty element not intersecting the boundary of $\mcal{P}$ contains exactly one element of $\frac{1}{s} G(s)$. Then, 
\begin{align*}
&\lim_{s \rightarrow\infty }\,\frac{m}{s^{d+e}} \! \sum_{x \in G(s)}   f(s,x) \\
&= \lim_{s \rightarrow\infty }\frac{m}{s^{d+e}}  \sum_{x \in G(s)} \big(s^e f_{e}(1,x/s) + s^{e-1} f_{e-1}(1,x/s) + \cdots + f_0(1,x/s) \big) \\
&=  \lim_{s \rightarrow\infty}\frac{m}{s^d}  \sum_{x \in \frac{1}{s}G(s) } \big( f_{e}(1,x) + s^{-1} f_{e-1}(1,x) + \cdots + s^{-e} f_0 (1,x) \big) \\
&= \lim_{s \rightarrow\infty} \frac{m}{s^d}  \sum_{x \in \frac{1}{s}\mcal{P}_s \cap \frac{1}{s}(L+\alpha)} \big( f_{e}(1,x) + s^{-1} f_{e-1}(1,x) + \cdots + s^{-e} f_0 (1,x) \big) \\
&= \lim_{s \rightarrow\infty} \frac{m}{s^d}
\sum_{ x \in \mcal{P} \cap \frac{1}{s}(L+\alpha) } \big( f_{e}(1,x) + s^{-1} f_{e-1}(1,x) + \cdots + s^{-e} f_0 (1,x) \big) \\
&= \lim_{s \rightarrow\infty} \frac{m}{s^d}
\sum_{ x \in \mcal{P} \cap \frac{1}{s}(L+\alpha) } f_{e}(1,x) \\
&= \lim_{s \rightarrow\infty} \sum_{ x \in \mcal{P} \cap \frac{1}{s}(L+\alpha) }  f_{e}(1,x) \vol\Big(\frac{1}{s} \Delta\Big) \\
&=\int_\mcal{P} f_e(1,x) \ddd{x}\,, \\
\end{align*}
where the last equality follows from noting that
\[ \sum_{ x \in \mcal{P} \cap \frac{1}{s}(L+\alpha) }  f_e(1,x) \vol\Big(\frac{1}{s} \Delta\Big) \]
 is a Riemann sum for $f_e(1,x)$ on $\mcal{P}$ with mesh equal to $\diam(\frac{1}{s}\Delta)$.
 \end{proof}

\section{$j$-multiplicity and degree of the fiber}\label{sec:j=te}

This section includes a crucial result (Theorem \ref{thm:j=te}) relating $j(I)$ and the degree of the fiber cone $e(\F(I))$ for particular ideals $I$. As a consequence we compute the $j$-multiplicity of any ideal defining a rational normal scroll.
 
\begin{thm}\label{thm:j=te}
Let $A=\bigoplus_{k\gs 0}\,A_k$ be a standard graded noetherian domain over a field $K=A_0$, $\mm=\bigoplus_{k>0}\,A_k$, and $I$ an ideal generated in a single degree $t\gs 1$ such that $\ell(I)=\dim\, A$. Then:
\begin{itemize}
\item[(i)] $j(I)$ is a positive integer multiple of $e(\F(I))$;
\item[(ii)] $j(I)\gs t\cdot e(\F(I))$, with equality if and only if $\R(I)_{\mm \R(I)}$ is a DVR.
\item[(iii)] If $[\big(I^s)^{\sat}]_r = [I^s]_r$\, for all $s \gg 0$ and all $r \gs st$, then 
\[j(I) = t\cdot e\big(\mathcal{F}(I)\big)\,.\]
\end{itemize}
\end{thm}

\begin{proof}
While showing (ii) we will show also (i). Let $R=\R(I)$, $G=\G(I)$, and $F=\F(I)$ denote respectively the Rees algebra, the associated graded ring, and the fiber cone of $I$. Let $a_1,\ldots, a_n$ be the minimal homogeneous generators of $I$. Then the fiber cone $F$ is isomorphic to $K[a_1,\ldots,a_n]\subset K[A_t]$  and hence a domain. Therefore $\mm G$ is a prime ideal of $G$ such that $G_{\mm G}$ is Artinian. Hence, for $u\gg 0$ we have that $\HH{0}{\mm G}{G}=0:_G \mm^u G$ and $\mm^uG_{mG}=0$. This implies that we have an isomorphism of $G_{\mm G}$-modules:
\begin{equation}\label{eqthm:j=te}
\HH{0}{\mm G}{G}_{\mm G}\cong \HH{0}{\mm G_{\mm G}}{G_{\mm G}}\cong G_{\mm G}\,.
\end{equation}
Thus we obtain from the additivity formula for the multiplicity that 
\[j(I)=e(\HH{0}{\mm G}{G})=e(F)\length(\HH{0}{\mm G}{G}_{\mm G})=e(F)\length(G_{\mm G})\,.\]
The above equality yields (i). In order to get (ii) we need to show that $\lambda(G_{\mm G})=\lambda(R_{\mm R}/IR_{\mm R})\gs t$. For this consider the filtration 
\[G_{\mm G}\supset \mm G_{\mm G} \supset \mm^2 G_{\mm G} \supset\cdots\,.\]
By Nakayama's lemma, two consecutive ideals of this sequence are distinct unless equal to 0. Now, if $\mm^{t-1}G_{mG}=\mm^{t-1}(R_{\mm R}/I R_{\mm R})=0$  then there is an element $f$ in $R\setminus\mm R$ such that $f\mm^{t-1}R\subset IR.$ Since $R$ has a bigraded structure, we can assume $f$ is homogenuous in this bigrading. By introducing an auxiliary variable $z$, write $R$ as $\bigoplus_{k\gs 0}I^kz^k$. In this notation, there exists $n\gs0$ such that $f=a z^n$ for some $a \in I^n\setminus \mm I^n$.  Then $az^n\mm^{t-1}\subset I^{n+1}z^n$, which implies $a\mm^{t-1}\subset I^{n+1}$. This is a contradiction since $a \mm^{t-1}$ has nonzero elements in degree $nt+(t-1)<(n+1)t.$ It follows that $\mm^uG_{mG}=\mm^u(R_{\mm R}/I R_{\mm R})\neq 0$ for $u<t$. 
Hence,
\begin{align*}
\lambda(G_{\mm G})&=\sum_{i=0}^{\infty} \lambda(\mm^i G_{\mm G}/\mm^{i+1}G_{\mm G})\\
&\gs \sum_{i=0}^{t-1} \lambda(\mm^i G_{\mm G}/\mm^{i+1}G_{\mm G})\gs t,
\end{align*}
which shows the inequality. 

If $R_{\mm R}$ is a DVR then $ \lambda(\mm^i G_{\mm G}/\mm^{i+1}G_{\mm G})=1$ for every $0\ls i<t$. Since $e\big(\mm R_{\mm R}\big)=1$ it follows from  \cite[Theorem 5.3]{KPU} that $\Quot\big( F\big)=\Quot\big( K[A_t] \big)$, i.e., every minimal generator of $\mm^t$ is a fraction of elements of $F$. The latter is equivalent to the existence of an integer $n\gs0$ and an element $g\in I^n\setminus\mm I^n$ such that $g\mm^t\subset I^{n+1}$, i.e., $\mm^tG_{mG}=0$. Hence the equality in (ii) holds. Conversely, the equality forces $\mm^tG_{mG}=0$ and likewise this implies $\Quot\big( F\big)=\Quot\big( K[A_t] \big)$ and $e\big(\mm R_{\mm R}\big)=1$. Since $\text{ht } I>0$, we have $\dim\, R=\dim\, A+1$. It follows $\dim\, R_{\mm R}=1$ and then $R_{\mm R}$ is a DVR.


To show (iii), we already have $j(I) \gs t e(F)$, so we only need to show the other inequality. 
Let $V(I,s)$ be the $K$-vector space $\big((I^{s+1})^{\sat}\cap I^s\big)/I^{s+1}$. The graded component in degree $r$ of $V(I,s)$ is 0 for all $r < st$. If $s$ is big enough, by assumption we also have
\[V(I,s)_r=0\,, \quad \forall r \gs (s+1)t\,.\]
Then $V(I,s)$ can be non-zero only in the $t$ degrees $ts, ts+1, \ldots, ts+t-1$. In these degrees, we have
\[ V(I,s)_r= [(I^{s+1})^{\sat}\cap I^s]_r \subset A_r\,, \]
so that, for an element $x$ of $A$ of degree one, multiplication by $x$ maps $V(I,s)_r$ injectively into $V(I,s)_{r+1}$ for all ${r=ts,\, \ldots ,\, ts+t-2}$. We then have the following chain of inequalities:
\[\rk_K\big(V(I,s)_{ts}\big) \ls \rk_K\big(V(I,s)_{ts+1}\big) \ls \cdots \ls \rk_K\big(V(I,s)_{ts+t-1}\big)\,.\]
Moreover,  $x$ multiplies  $V(I,s)_{ts+t-1}$ injectively into $[I^{s+1}]_{(s+1)t}$\,. This implies that 
\[\rk_K\big(V(I,s)_{ts+t-1}\big) \ls \rk_K[I^{s+1}]_{(s+1)t}=\rk_K (F_{s+1})\,,\] 
the latter being equal to the Hilbert function of $F$ evaluated at $s+1$. Taking the suitable limits, this implies at once that $j(I)\ls t\cdot e(F)$, as desired. 
\end{proof}

Related formulas expressing $e\big(\F(I)\big)$ in terms of multiplicities of $I$ can be found in \cite[Section 6]{USV} and \cite[Section 5]{Xie}. The assumption${[\big(I^s)^{\sat}]_r = [I^s]_r}$ in (iii) is rather strong, but essential for the theorem as observed in the following example.  

\begin{example}
Let $A=K[x,y]$ and $I=(x^2,y^2)$. We have $\mathcal{F}(I)=K[x^2,y^2]$, so $e(\mathcal{F}(I))=1$. Then $t \cdot e(\mathcal{F}(I))=2$, but $j(I)=e(I)=4$. Note that for all $s$ we have $xy^{2s-1}\in [\big(I^s)^{\sat}]_{2s} \setminus [I^s]_{2s}$.
\end{example}
We will see throughout the rest of the paper that the third point of the above theorem is a very powerful tool, as some good classes of ideals satisfy the hypotheses. In particular, if $I$ has linear powers (i.e., each power of $I$ has a linear resolution), then $I$ satisfies the condition of (iii), since ${\reg(R/I^s)=st-1}$, so 
\[{\big[(I^s)^{\sat}/I^s \big]_{\gs st}=\big[\HH{0}{\mm}{R/I^s}\big]_{\gs st} = 0}\,.\]
Examples of classes of ideals having linear powers are:

\begin{itemize}
\item[i)] Ideals of maximal minors of a sufficiently general matrix with linear entries (\cite[Theorem~3.7]{BCV}).
\item[ii)] Ideals defining varieties of minimal degree (\cite{BCV}).
\item[iii)] Edge ideals of graphs whose compliment is chordal (\cite[Theorem~3.2]{HHZ}). 
\item[iv)] Ideals generated by monomials corresponding to the bases of a matroid (\cite[Theorems~5.2,~5.3]{CH}). 
\item[v)] Stable ideals generated in one degree (\cite{CH}).
\end{itemize}

\subsection{The $j$-multiplicity of rational normal scrolls}\label{subrat}

Given positive integers $a_1\ls \cdots \ls a_d$, the associated $d$-dimensional rational normal scroll is the projective subvariety of $ \PP^N$, where ${N=\sum_{i=1}^d a_i+d-1}$, defined by the ideal 
\[I=I(a_1,\ldots ,a_d)\subseteq K[x_{i,j}: i=1,\ldots ,d; j=1,\ldots ,a_i+1]\] 
generated by the $2$-minors of the matrix
\[\begin{pmatrix}
x_{1,1} & \cdots & x_{1,a_1} & x_{2,1} & \cdots & x_{2,a_2} & \cdots & x_{d,1} & \cdots & x_{d,a_d}\\
x_{1,2} & \cdots & x_{1,a_1+1} & x_{2,2} & \cdots & x_{2,a_2+1} & \cdots & x_{d,2} & \cdots & x_{d,a_d+1}
\end{pmatrix}\,.\]
In \cite[Corollary~3.9]{BCV}, it has been proved that for all positive integers $s$ and ${a_1\ls \cdots \ls a_d}$, the ideals $I^s$ have a $2s$-linear free resolution. It follows, as noted above, that
\[[(I^s)^{\sat}]_r=[I^s]_r\,, \quad \forall r\gs 2s\,.\]
Furthermore, \cite[Theorem~3.7(2)]{BCV} implies that the Betti numbers $\beta_i(I^s)$ depend only on $d$, the sum $\sum_{i=1}^da_i=c$, and $s$. In particular the Hilbert function of the fiber cone 
\[\rk_K(\F(I)_s)=\beta_0(I^s)\]
depends only on $d$ and $c$. For fixed values of these invariants, there is a unique rational normal scroll associated to $b_1\ls \cdots \ls b_d$, where ${\sum_{i=1}^db_i=c}$ and ${b_d\ls b_1+1}$ (such a rational normal scroll is called balanced). The ideal $I(b_1,\ldots ,b_d)$ can be expressed as the ideal generated by the $2$-minors of a $2\times c$ extended Hankel matrix of pace $d$ and this fact allowed the authors of \cite{CHV} to find a quadratic Gr\"obner basis for the ideal defining $\F(I(b_1,\ldots ,b_d))$. As a consequence, they computed the dimension and the degree of the fiber cone of $I(b_1,\ldots ,b_d)$ in \cite[Corollary~4.2, Corollary~4.5]{CHV}. Concluding, we infer:

\begin{thm}\label{j-scrolls}
Let $a_1\ls \cdots \ls a_d$ be positive integers and put $c= \sum_{i=1}^da_i$. Then the $j$-multiplicity of $I(a_1,\ldots ,a_d)$ is:
\begin{align*}
\begin{cases}
0 & \textit{ if }c<d+3\,, \\
\displaystyle 2\cdot \left(\binom{2c-4}{c-2}-\binom{2c-4}{c-1}\right) & \textit{ if }c=d+3 \,,\\
\displaystyle 2\cdot \left(\sum_{j=2}^{c-d-1} \binom{c+d-1}{c-j} - \binom{c+d-1}{c-1}(c-d-2)\right) & \textit{ if }c>d+3\,.
\end{cases}
\end{align*}
\end{thm}

\begin{example} From the theorem above, we have that
\[j(I(4))=j\Big(I_2 \begin{pmatrix}
x_1 & x_2 & x_3 & x_4 \\
x_2 & x_3 & x_4  & x_5
\end{pmatrix}\Big) =4\,,\]
and 
\[j(I(3,2))=j\Big(I_2 \begin{pmatrix}
x_1 & x_2 & x_3 & x_5 & x_6 \\
x_2 & x_3 & x_4  & x_6 & x_7
\end{pmatrix}\Big) = 10\,.\]
Our values agree with those of Nishida and Ulrich in \cite[Example~4.8]{NU}, computed there for the same ideals by using the length formula.
\end{example}

\section{Standard monomial theory}\label{sec:standardmon}

In the rest of the paper we will focus on the computation of the discussed multiplicities for determinantal ideals $I$. Such ideals satisfy the property of the third point in Theorem \ref{thm:j=te}, but this time this is useless: the fiber cone of $I$ is the algebra of minors discussed in \cite{BC}, and it is quite subtle (for example to determine its defining equations seems a very difficult problem \cite{BCV1}). Also the degree of this algebra is unknown, as stated at the end of the introduction in \cite{BC}. For these reasons we will develop a technique allowing us to compute simultaneously the $j$-multiplicity and the $\epsilon$-multiplicity of such ideals (thereby we will also get the degree of the algebras of minors).

In this section we are going to provide the necessary information on the structure of powers of the ideals generated by the minors or by the pfaffians of the following matrices:
\begin{itemize}
\item[(i)] an $m\times n$ matrix of indeterminates $X$;
\item[(ii)] an $n\times n$ symmetric matrix of indeterminates $Y$;
\item[(iii)] an $n\times n$ skew-symmetric matrix of indeterminates $Z$.
\end{itemize} 
The first proofs of the results we are going to summarize were given in characteristic $0$, where representation theoretic tools are available (see \cite{DEP} for (i), \cite{Ab} for (ii) and \cite{AD} for (iii)). For arbitrary fields one has to find new proofs, reformulating the statements in terms of {\it standard monomials}, which form a particular $K$-basis of the polynomial rings containing the above ideals. That the standard monomials form a $K$-basis was proved in \cite{DRS} for (i), and in \cite{DP} for (ii) and (iii); how to use these tools to understand the powers of the ideals that we are interested in (regardless to the characteristic of the base field) was explained in \cite{Bruns} for (i) and in \cite{De} for (iii). For (ii) we could not find any reference; however this case is completely analogous to (i), and we will indicate the essential steps to prove the needed statements. Two excellent sources for the standard monomial theory are \cite{DEP2} and \cite{BV}.   

\subsection{The generic case.} Let $X=(x_{ij})$ be an $m\times n$-matrix whose entries are indeterminates over $K$. 
A $k$-minor is the determinant of a $k\times k$-submatrix of $X$, and the usual notation for it is 
\[[i_1,\dots,i_k|j_1,\dots,j_k]:=\det\begin{pmatrix}
x_{i_1,j_1}&\dots&x_{i_1,j_k}\\
\vdots&&\vdots\\
x_{i_k,j_1}&\dots&x_{i_k,j_k}
\end{pmatrix}\in K[X]\,,
\]
so that $i_1,\dots,i_k$ denote the rows of the minor and $j_1,\dots,j_k$ the columns of the minor. For $t\ls \min\{m,n\}$ the algebraic variety of $\mathbb{A}_K^{mn}$ consisting of the $m\times n$ matrices of rank at most $t-1$ is cut out by the prime ideal $I_t(X)\subseteq K[X]$ generated by the $t$-minors of $X$. To study these ideals it is convenient to consider the set of minors (of any size) with the following partial order:
\begin{align*}
[i_1,\dots,i_k|j_1,\dots,j_k]\ls [u_1,\dots,u_h|v_1,\dots,v_h] \ \iff \\
k\gs h, \ i_q\ls u_q, \ j_q\ls v_q \ \forall \ q\in\{1,\ldots ,h\}\,.
\end{align*}
A {\it standard monomial} is a product of minors $\delta_1 \cdots \delta_p$ such that ${\delta_1\ls \cdots \ls \delta_p}$. It turns out that the standard monomials form a $K$-basis for $K[X]$. If $\Delta$ is the product of minors $\delta_1\cdots \delta_p$ where $\delta_i$ is an $a_i$-minor, then the vector ${(a_1,\ldots ,a_p)\in \NN^p}$ is referred to be the {\it shape} of $\Delta$.

It is customary to associate to a standard monomial a pair of tableaux. Recall that a \emph{Young diagram} is a collection of boxes aligned in rows and columns, starting from the left in each row, such that the number of boxes in row $i$ is not less than the number of boxes in row $i+1$. A \emph{Young tableau (on $n$)} is a Young diagram filled in with numbers in $\{1,\dots,n\}$. A Young tableau is called \emph{standard}  if the filling is such that the entries in each row are strictly increasing and the entries in each column are nondecreasing. The \emph{shape} of a tableau or diagram is the list $(a_1,\dots,a_p)$, where $a_i$ is the number of boxes in row $i$. We assign to each standard monomial $\Delta=\delta_1\cdots \delta_p$ a pair of tableaux as follows: for the diagram of each, use the diagram with the same shape as $\Delta$. In the first diagram, in row $i$, list the rows of $\delta_i$. In the second diagram, in row $i$, list the columns of $\delta_i$. For example,

\[
{\setlength{\unitlength}{0.8mm}
\begin{picture}(30,15)(-5,5)

\put(-47,13.0){$\bigg($}

\put(-40,20){\line(1,0){15}} \put(-40,15){\line(1,0){15}}
\put(-40,10){\line(1,0){5}}

\put(-25,20){\line(0,-1){5}} \put(-30,20){\line(0,-1){5}}
\put(-35,20){\line(0,-1){10}} \put(-40,20){\line(0,-1){10}}

\put(-21,15){,}

\put(-15,20){\line(1,0){15}} \put(-15,15){\line(1,0){15}}
\put(-15,10){\line(1,0){5}}

\put(-15,20){\line(0,-1){10}} \put(-10,20){\line(0,-1){10}}
\put(-5,20){\line(0,-1){5}} \put(0,20){\line(0,-1){5}}

\put(3,13.0){$\bigg)$}

\put(-28.5,16){$4$} \put(-33.5,16){$3$} \put(-38.5,16){$1$}
\put(-38.5,11){$3$} \put(-13.5,16){$2$} \put(-8.5,16){$3$}
\put(-3.5,16){$5$} \put(-13.5,11){$2$}

\put(10,14){$\leftrightarrow$}

\put(20,13.5){$[1,3,4|2,3,5] \cdot [3|2]\,.$}

\end{picture}}
\]
The condition that the product of minors $\Delta$ is a standard monomial is equivalent to the labelling of the diagrams to form standard tableaux.

The good news is that to detect if $\D$ belongs to the symbolic powers or to (the integral closure of) the ordinary powers of $I_t(X)$ it is enough to look at the shape of $\D$. Precisely, by \cite[Theorem~10.4]{BV} we have:

\begin{prop}\label{thm2_generic}
With the above notation:
\[\D\in I_t(X)^{(r)}\iff a_i\ls \min\{m,n\} \ \forall \ i \mbox{ \ \ and \ \ }\sum_{i=1}^p\max\{0,a_i-t+1\}\gs r\,.\]
\end{prop}

Furthermore, by \cite[Theorems~1.1~and~1.3]{Bruns}:

\begin{thm}\label{thm1_generic}
If $\chara(K)=0$ or $>\min\{t,m-t,n-t\}$, then $I_t(X)^s$ has primary decomposition:
\[I_t(X)^s=\bigcap_{j=1}^t I_j(X)^{((t-j+1)s)}\,.\]
In arbitrary characteristic such a decomposition holds true for integral closures:
\[\overline{I_t(X)^s}=\bigcap_{j=1}^t I_j(X)^{((t-j+1)s)}\,.\]
\end{thm}

\subsection{The symmetric case.} Let $Y=(y_{ij})$ be an $n\times n$ symmetric matrix (meaning ${y_{ij}=y_{ji}}$) whose entries are indeterminates over $K$. A standard monomial theory is available also in this situation, and is useful for studying the prime ideal $I_t(Y)\subseteq K[Y]$ defining the locus of symmetric matrices of rank at most $t-1$: The minors are denoted as before. A minor $[i_1,\dots,i_k|j_1,\dots,j_k]$ is called a {\it doset minor} if $i_q\ls j_q$ for all $q=1,\ldots ,k$. In the symmetric case the doset minors are the only relevant ones (the doset minors of size $t$ are already enough to generate $I_t(Y)$), and they naturally form a poset by the following order:	

\begin{align*}
[i_1,\dots,i_k|j_1,\dots,j_k]\ls [u_1,\dots,u_h|v_1,\dots,v_h] \ \iff \\
k\gs h, \ j_q\ls u_q \ \forall \ q\in\{1,\ldots ,h\}.
\end{align*}

Again, a product of doset minors $\D =\delta_1 \cdots \delta_p$ is a standard monomial if $\delta_1\ls \cdots \ls \delta_p$, and the standard monomials form a $K$-basis of $K[Y]$. If the shape of $\D$, defined as before, is the vector $(a_1,\ldots ,a_p)\in \NN^p$, similarly to the generic case we have:

\begin{prop}\label{thm2_symmetric}
With the above notation we have:
\[\D\in I_t(Y)^{(r)}\iff a_i\ls n \ \forall \ i \mbox{ \ \ and \ \ }\sum_{i=1}^p\max\{0,a_i-t+1\}\gs r\,.\]
\end{prop}
\begin{proof}
Arguing as in the proof of \cite[Lemma~10.3]{BV}, one shows that $y_{nn}$ is a non-zerodivisor for $K[Y]/J(t,r)$, where $J(t,r)$ is the ideal generated by the products of doset minors whose shape $(a_1,\ldots ,a_p)$ satisfies the condition ${\sum_{i=1}^p\max\{0,a_i-t+1\}\gs r}$. At this point we can exploit the ``localization trick" at $y_{nn}$ and apply the induction on $t$ to show $J(t,r)=I_t(Y)^{(r)}$, as explained before \cite[Lemma~10.3]{BV}.
\end{proof}

Also in the symmetric case we have a nice primary decomposition for the powers, namely:

\begin{thm}\label{thm1_symmetric}
If $\chara(K)=0$ or $>\min\{t,n-t\}$, then $I_t(Y)^s$ has primary decomposition:
\[I_t(Y)^s=\bigcap_{j=1}^t I_j(Y)^{((t-j+1)s)}\,.\]
In arbitrary characteristic such a decomposition holds true for integral closures:
\[\overline{I_t(Y)^s}=\bigcap_{j=1}^t I_j(Y)^{((t-j+1)s)}\,.\]
\end{thm}
\begin{proof}
The first part can be derived as in \cite[Corollary~10.13]{BV}, the second as in the proof of \cite[Theorem~1.3]{Bruns} (it is enough to specialize to symmetric matrices the equations of \cite[Lemma~1.4]{Bruns} and \cite[Lemma~1.5]{Bruns}).
\end{proof}

As in the generic case, we can identify standard monomials with Young tableaux. In this case however, the restriction to doset minors and the more stringent partial ordering make the appropriate object to identify with a standard monomial a single tableau, as opposed to a pair in the generic case. Specifically, if a standard monomial $\D = \delta_1\delta_2 \cdots \delta_p$ has shape $(a_1,a_2,\ldots ,a_p)$, we associate a diagram $D$ of shape $(a_1, a_1, a_2, a_2, \dots, a_p, a_p)$. The rows of $\delta_i$ are listed in row $2i-1$ of $D$, and the columns of $\delta_i$ are listed in row $2i$ of $D$. For example,

\[
{\setlength{\unitlength}{0.8mm}
\begin{picture}(30,15)(-5,5)

\put(-30,20){\line(1,0){15}} \put(-30,15){\line(1,0){15}}
\put(-30,10){\line(1,0){15}} \put(-30, 5) {\line(1,0){5}}
\put(-30, 0) {\line(1,0){5}}

\put(-30,20){\line(0,-1){20}} \put(-25,20){\line(0,-1){20}}
\put(-20,20){\line(0,-1){10}} \put(-15,20){\line(0,-1){10}}

\put(-28.5,16){$1$} \put(-23.5,16){$3$} \put(-18.5,16){$4$} 
\put(-28.5,11){$2$} \put(-23.5,11){$4$} \put(-18.5,11){$5$}
\put(-28.5, 6){$2$} 
\put(-28.5, 1){$4$} 

\put(0,10){$\leftrightarrow$}

\put(15,10){$[1,3,4|2,4,5] \cdot [2|4]\,.$}

\end{picture}}
\]

\

\subsection{The alternating case.} Now let $Z=(z_{ij})$ be an $n\times n$ skew-symmetric matrix ($z_{ij}=-z_{ji}$) whose nonzero entries are indeterminates over $K$. Here the situation is a bit different by the previous ones, since the ideal $I_t(Z)$ is not radical. This is because any minor of the form $[i_1,\ldots ,i_{2k}|i_1,\ldots ,i_{2k}]$ is a square of a polynomial, namely its {\it pfaffian}, which we will simply denote by $[i_1,\ldots, i_{2k}]$. Indeed the locus of the alternating matrices of rank at most $2t-2$ is the same as the locus of the alternating matrices of rank at most $2t-1$, which is defined by the prime ideal $P_{2t}(Z)\subseteq K[Z]$ generated by all the $2t$-pfaffians $[i_1,\ldots ,i_{2t}]$. We can equip the set of pfaffians with the following partial order: 
\begin{align*}
[i_1,\dots,i_{2k}]\ls [u_1,\dots,u_{2h}] \ \iff \\
k\gs h, \ i_q\ls u_q \ \forall \ q\in\{1,\ldots ,2h\}\,.
\end{align*}

Once again, a product of pfaffians $\D =\delta_1 \cdots \delta_p$ is a standard monomial if $\delta_1\ls \cdots \ls \delta_p$, and these standard monomials are a $K$-basis of $K[Z]$. The shape of $\D$ will be the vector $(a_1,\ldots ,a_p)\in \NN^p$ where $\delta_i$ is a $2a_i$-pfaffian. By \cite[Theorem~2.1]{De} we have:

\begin{prop}\label{thm2_pfaffians}
With the above notation:
\[\D\in P_{2t}(Z)^{(r)}\iff a_i\ls \lfloor n/2\rfloor \ \forall \ i \mbox{ \ \ and \ \ }\sum_{i=1}^p\max\{0,a_i-t+1\}\gs r\,.\]
\end{prop}

Furthermore, by \cite[Proposition~2.6]{De} (for the second part of the statement see the comment below \cite[Proposition~2.6]{De}):

\begin{thm}\label{thm1_pfaffians}
If $\chara(K)=0$ or $>\min\{2t,n-2t\}$, then $P_{2t}(Z)^s$ has primary decomposition:
\[P_{2t}(Z)^s=\bigcap_{j=r}^t P_{2j}(Z)^{((t-j+1)s)} \ \ \mbox{ where } \ \ r=\max\{1,\lfloor n/2\rfloor - s(\lfloor n/2\rfloor -t)\}\,.\]
In arbitrary characteristic such a decomposition holds true for integral closures:
\[\overline{P_{2t}(Z)^s}=\bigcap_{j=r}^t P_{2j}(Z)^{((t-j+1)s)} \ \ \mbox{ where } \ \ r=\max\{1,\lfloor n/2\rfloor - s(\lfloor n/2\rfloor -t)\}\,.\]\end{thm}

As above, we associate to each pfaffian a tableau: to a pfaffian of shape $(a_1,\ldots ,a_p)$ we associate the tableau of shape $(2 a_1,\ldots ,2 a_p)$ in which the $i^{\text{th}}$ row lists the rows of the $i^{\text{th}}$ minor.

\section{Counting Young Tableaux}\label{sec:count}

\begin{prop}[(The Diagram Criterion, Generic Case)]\label{diagram_generic} Let $X$ be a matrix of generic indeterminates. A standard monomial of shape $(a_1,\dots,a_p)$ belongs to 
\[(\overline{I_t(X)^{s+1}})^{\sat} \setminus \overline{I_t(X)^{s+1}}\]
if and only if its shape satisfies the conditions
\begin{itemize}
\item[1)] $a_i\ls m$ for all $i$,
\item[2a)] $\sum{a_i} < t(s+1)$, and
\item[3)] $p \ls \sum{a_i} - (t-1)(s+1)$.
\end{itemize}
A standard monomial belongs to 
\[\big((\overline{I_t(X)^{s+1}})^{\sat} \cap \overline{I_t(X)^s}\,\big) \setminus \overline{I_t(X)^{s+1}}\]
if and only if its shape satisfies the conditions (1) \& (3) above plus the condition
\begin{itemize}
\item[2b)] $ts \ls \sum{a_i} < t(s+1)$.
\end{itemize}
\end{prop}
\begin{proof}
By Theorem~\ref{thm1_generic}, we have
\[(\overline{I_t(X)^{s+1}})^{\sat} \setminus \overline{I_t(X)^{s+1}} = \bigcap_{j=2}^t I_j(X)^{((t-j+1)(s+1))}\setminus \bigcap_{j=1}^t I_j(X)^{((t-j+1)(s+1))}\,.\]
Applying Theorem~\ref{thm2_generic}, a standard monomial is in this set if and only if its shape satisfies the conditions
\begin{itemize}
\item[1)] $a_i\ls m$ for all $i$,
\item[2a)] $\sum{a_i} < t(s+1)$, and
\item[3')] $\sum{\max\{0, a_i-j+1\}}\gs (t-j+1)(s+1)$ for all $j=2,\dots,t$.
\end{itemize}
For $j=1,\dots,m$, let $r_j$ denote $|\{i:a_i=j\}|$. We compute
\begin{align*}
\sum_{i=1}^p &\max\{0, a_i-j+1\} =
\sum_{i=1}^{p-(r_1+\cdots+r_{j-1})}{a_i}-(j-1)(p-(r_1+\cdots+r_{j-1}))\\
&=\sum_{i=1}^p a_i - (r_1 + 2 r_2 + \cdots + (j-1) r_{j-1}) - (j-1)(p-(r_1+\cdots+r_{j-1}))\\
&=\sum_{i=1}^p a_i - (j-1)p + (j-2)r_1 + (j-3) r_2 + \cdots + r_{j-2}\,.
\end{align*}
Then,
\[
\sum_{i=1}^p \max\{0, a_i-j+1\} \gs (t-j+1)(s+1) \qquad \text{for all $j=2,\dots,t$}
\]
if and only if
\[
 p \ls s+1 + \min_{2\ls j \ls t} \!{\bigg(\frac{\sum{a_i} - t(s+1)}{j-1} + \frac{(j-2)r_1 + (j-3) r_2 + \cdots + r_{j-2}}{j-1}\bigg)}\,.
\]
Since the first fraction above is negative and the second positive, the minimum is achieved for $j=2$. This shows that condition (3') can be replaced by (3).

For the second statement we note that a standard monomial is in 
\[\big((\overline{I_t(X)^{s+1}})^{\sat} \cap \overline{I_t(X)^s}\,\big) \setminus \overline{I_t(X)^{s+1}}\]
if and only if it is in
\[(\overline{I_t(X)^{s+1}})^{\sat} \setminus \overline{I_t(X)^{s+1}}\]
and in $I_1(X)^{(ts)}$.
\end{proof}

There are nearly identical characterizations in the symmetric and pfaffian cases:

\begin{prop}[(The Diagram Criterion, Symmetric Case)]\label{diagram_symmetric} 
Let $Y$ be a generic symmetric matrix of indeterminates.
A standard monomial $M$ of shape $(a_1, a_2, \dots, a_p)$ belongs to 
\[(\overline{I_t(Y)^{s+1}})^{\sat} \setminus \overline{I_t(Y)^{s+1}}\]
if and only if conditions (1), (2a), and (3) of Proposition~\ref{diagram_generic} hold for 
\[(a_1,\dots,a_{p})\,;\]
a such standard monomial corresponds to a diagram, the shape of which is $(a_1,a_1,a_2,a_2,\dots,a_p,a_p)$. The monomial $M$ belongs to 
\[\big((\overline{I_t(Y)^{s+1}})^{\sat} \cap \overline{I_t(Y)^s}\,\big) \setminus \overline{I_t(Y)^{s+1}}\]
if and only if conditions (1), (2b), and (3) of Proposition~\ref{diagram_generic} hold for the same $(a_1,\dots,a_{p})$.
\end{prop}

\begin{prop}[(The Diagram Criterion, pfaffian Case)]\label{diagram_pfaffian} 
Let $Z$ be a generic antisymmetric matrix of indeterminates.
A standard monomial $M$ of shape $(a_1, a_2,\dots, a_p)$ belongs to 
\[(\overline{P_{2t}(Z)^{s+1}})^{\sat} \setminus \overline{P_{2t}(Z)^{s+1}}\]
if and only if conditions (1), (2a), and (3) of Proposition~\ref{diagram_generic} hold for

\[(a_1,a_2,\dots,a_p)\,,\] 
with $m=\lfloor \frac{n}{2}\rfloor $; a such standard monomial represents a diagram of shape $(2a_1,\dots,2a_{p})$. The monomial $M$ belongs to 
\[\big((\overline{P_{2t}(Z)^{s+1}})^{\sat} \cap \overline{P_{2t}(Z)^s}\,\big) \setminus \overline{P_{2t}(Z)^{s+1}}\]
if and only if conditions (1), (2b), and (3) of Proposition~\ref{diagram_generic} hold for the same $(a_1,\dots,a_{p})$ with $m=\lfloor \frac{n}{2}\rfloor$.
\end{prop}

The number of standard tableaux on $\{1,\ldots ,n\}$ of shape $(a_1,\ldots ,a_p)$ is the dimension of the irreducible (if $\chara(K)=0$) $\GL_n(K)$-representation associated to $(a_1,\ldots ,a_p)$. This dimension can be computed by the hook-length formula; however, for our aims it is convenient to have a polynomial formula in suitable data characterizing the given shape, while the hook-length formula is certainly not polynomial in nature. 

\begin{remark}
The notation we are using are dual to the standard ones used in representation theory: for example, with our conventions, to the diagram $(t)$ would correspond $\bigwedge^t K^n$.
\end{remark}

To this goal, we set $r_i$ to be the number of rows of the diagram $(a_1,\ldots ,a_p)$ with exactly $i$ boxes. If we deal with diagrams with at most $m$ columns, then the diagram will be determined by $r_1,\ldots ,r_m$. The next proposition provides a polynomial formula for the number of standard tableaux on $\{1,\ldots ,n\}$ of shape $(a_1,\ldots ,a_p)$ in terms of the $r_i$. Such a quantity will be denoted by $W_n(r_1,\ldots ,r_m)$, and in particular it will be a (inhomogeneous) polynomial of degree $m(n-m)+\binom{m}{2}$. 

\begin{prop}\label{W-prop}
For all $i=1,\ldots ,m$, set $B_i=r_m+\cdots +r_{m-i+1}$. Then:
\[
\displaystyle W_n(r_1,\ldots, r_m)= \frac{ \prod_{i=1}^{m}(B_{i}+i)\cdots(B_{i}+i+n-m-1)\cdot \prod_{i<j}(B_{j}-B_i+j-i)}{(n-1)!(n-2)!\cdots (n-m)! }\,.
\]
\end{prop}

\begin{proof}

The proof is an application of the hook length formula \cite{Sta1}. 

We will divide the diagram in $\frac{m(m+1)}{2}$ regions $R_{a,b}$, for all $1\ls a\ls m$ and ${1\ls b\ls (m+1-a)}$. The boxes of $R_{a,b}$ are the ones $(i,b)$ such that ${B_{a-1}+1\ls i\ls B_a}$ where $B_0=0$. 

The factors in the hook length formula corresponding to the boxes in the region $R_{a,b}$ are:
\[\prod_{i=B_{a-1}+1}^{B_a}\frac{n-b+i}{(m+2-a-b)+B_{m+1-b}-i}\,.\]

Now we multiply together the factors corresponding to a fixed $b$, that is to a fix column:
\begin{align*}
\prod_{i=1}^{B_1}\frac{n-b+i}{(m+1-b)+B_{m+1-b}-i}
\cdot&\prod_{i=B_1+1}^{B_2}\frac{n-b+i}{(m-b)+B_{m+1-b}-i}
\\
\cdots &\prod_{i=B_{m-b}+1}^{B_{m+1}-b}\frac{n-b+i}{1+B_{m+1-b}-i}\,.
\end{align*}
Now, we can rearrange these factors to obtain:
\begin{align*}
=\prod_{i=1}^{B_{m+1-b}}\frac{n-b+i}{(m+1-b)+B_{m+1-b}-i}
&\cdot \prod_{i=B_1+1}^{B_{m+1-b}}\frac{(m+1-b)+B_{m+1-b}-i}{(m-b)+B_{m+1-b}-i}\\
\prod_{i=B_2+1}^{B_{m+1-b}}\frac{(m-b)+B_{m+1-b}-i}{(m-1-b)+B_{m+1-b}-i}
&\cdots
\prod_{i=B_{m-b}+1}^{B_{m+1-b}}\frac{2+B_{m+1-b}-i}{1+B_{m+1-b}-i}\,.
\end{align*}

Most of these terms cancel and this product is equal to:
\begin{align*}
=\frac{(n-b+B_{m+1-b})\ldots (m+1-b+B_{m+1-b})}{(n-b)\cdots 
(m+1-b)}&\cdot \frac{m-b+B_{m+1-b}-B_1}{m-b}\cdot\\
\frac{m-1-b+B_{m+1-b}-B_2}{m-1-b}&\cdots \frac{1+B_{m+1-b}-B_{m-b}}{1}
\end{align*}
\[
=\frac{\prod_{i=m+1-b}^{n-b} (i+B_{m+1-b}) \cdot \prod_{i=1}^{m-b} \big( (m+1-b-i)+B_{m+1-b}-B_i  \big) }{(n-b)!}\,. 
\]

Multiplying this last expression over all $b$ we obtain the desired conclusion.
\end{proof}

\section{Multiplicities of determinantal varieties}\label{section_main}

In this section we give expressions for the $j$-multiplicity, the $\e$-multiplicity, and the multiplicity of the fiber cone of any determinantal ideal of generic, generic symmetric, and generic skew-symmetric matrices. Let $\nu$ be the measure on the affine subspace $\sum{z}=t$ such that $\pi_* \nu$ is Lebesgue measure, where $\pi$ is projection onto one of the coordinate hyperplanes.

\begin{thm}\label{thm_generic} Let $X$ be a generic $m \times n$ matrix of indeterminates, $t$ an integer with $0<t<m$, and set
\[
c =\frac{(mn-1)!}{(n-1)!(n-2)!\cdots (n-m)!\cdot m!(m-1)!\cdots 1!}\,.
\]
Then
\begin{itemize}
\item[(i)]\qquad $\displaystyle
j(I_t(X))= {ct} \int\limits_{\substack{[0,1]^m \\ \sum{z_i}= t}}(z_1\cdots z_{m})^{n-m}\prod_{1\ls i<j \ls m}(z_j-z_i)^2 \ddd{\nu}\,;$\vspace{2mm}
\item[(ii)]\qquad $\displaystyle
\e(I_t(X))=\,\,cmn\!\!\!\!\!\!\!\!\!\!\!\!\!\!\!\!\!\!\!\!\!\!\!\!\int\limits_{\substack{[0,1]^m \\  \max_i\{z_i\}+t-1\ls \sum{z_i} \ls t }}\!\!\!\!\!\!\!\!\!\!\!\!\!\!\!\!\!\!\!\!\!(z_1\cdots z_{m})^{n-m}\prod_{1\ls i<j \ls m}(z_j-z_i)^2 \ddd{z}\,;$\\
\item[(iii)] If $\chara(K)=0$ or $>\min\{t,m-t,n-t\}$ then,

\qquad $\displaystyle
e(A_t(X))=c\int\limits_{\substack{[0,1]^m \\ \sum{z_i}= t}}(z_1\cdots z_{m})^{n-m}\prod_{1\ls i<j \ls m}(z_j-z_i)^2 \ddd{\nu}\,;$\vspace{2mm}
\vspace{2mm}
\end{itemize}
where $A_t(X)$ is the algebra of minors $\F(I_t(X))$.
\end{thm}

\begin{proof}
For (i), by Proposition \ref{j-prop} we have \[\displaystyle j(I_t(X))=\displaystyle\lim_{s\rightarrow \infty} \frac{(d-1)!}{s^{d-1}}\length\Big(\big((\overline{I_t(X)^{s+1}})^{\sat} \cap \overline{I_t(X)^s}\,\big) / \overline{I_t(X)^{s+1}}\,\Big)\,.\]
We compute the length 
\[\ell_t(s):=\lambda\Big(\big((\overline{I_t(X)^{s+1}})^{\sat} \cap \overline{I_t(X)^s}\,\big) / \overline{I_t(X)^{s+1}}\,\Big)\] 
by counting the number of standard monomials in
 \[\big((\overline{I_t(X)^{s+1}})^{\sat} \cap \overline{I_t(X)^s}\,\big) \setminus \overline{I_t(X)^{s+1}}\,.\]
Write $G_t(s)$ for the set of diagrams corresponding to these standard monomials. As in the discussion before Proposition \ref{W-prop}, for a standard monomial denote by $r_i$ the number of rows in the diagram associated to it. The condition in Proposition~\ref{diagram_generic} can be expressed in terms of the numbers $r_i$. Set 
\[k=\sum a_i-ts= mr_m+(m-1)r_{m-1}+\cdots+r_1 -ts \,.\] 
A diagram belongs to $G_t(s)$ if and only if 
\begin{itemize}
\item[1)] $r_i=0$, $\forall i> m$
\item[2b)] $0\ls k<t$
\item[3)] $r_m+r_{m-1}+\cdots+r_1\ls s-t+1+k\,,$
\end{itemize}
which can be rewritten as
\begin{equation*}
\begin{aligned}
 0\ls k&<t\\
 (m-1)r_{m-1}+(m-2)r_{m-2}+\cdots+r_1&\ls st+k\\
 (m-1)r_{m-1}+(m-2)r_{m-2}+\cdots+r_1&\equiv st+k \mdd{m}\\
 r_{m-1}+2r_{m-2}+\cdots+(m-1)r_1&\ls s(m-t)+m(1-t)+k(m-1)\,.
\end{aligned}
\end{equation*}
Then, 
\[\ell_t(s)=\sum_{(r_{m-1},\dots,r_1)\in G_t(s)}{W_m (r_1,\ldots,r_m) W_n (r_1,\ldots,r_m)}\,,\]
with $r_m=\frac{1}{m}(st+k-(m-1)r_{m-1}-\cdots-r_1)$. We may now apply Lemma~\ref{integration} to obtain
\[\lim_{s\rightarrow \infty} \frac{\ell_t(s)}{s^{mn-1}}=\frac{t}{m}\int_\P \frac{(B_1\cdots B_m)^{n-m} \prod_{i<j}(B_j-B_i)^2}{(n-1)!\cdots(n-m)!\cdot (m-1)!\cdots 1!}\ddd{r_1}\cdots\ddd{r_{m-1}}\,,\]
for $B_k=r_m+\cdots+r_{m-k+1}$, and $\P\subset \RR^{m-1}$ defined by the inequalities
\begin{equation*}
\begin{aligned}
 r_i &\gs 0\\
 (m-1)r_{m-1}+(m-2)r_{m-2}+\cdots+r_1&\ls t\\
 r_{m-1}+2r_{m-2}+\cdots+(m-1)r_1&\ls m-t\,,
\end{aligned}
\end{equation*}
and $r_m=\frac{1}{m}(t-(m-1)r_{m-1}-\cdots-r_1)$. Then, applying the change of variables $z_i=B_i,$ for $i=\{1,\dots,m\}$, one has
\[
\begin{aligned}
j(I_t(X))&=\lim_{s\rightarrow \infty} \frac{(nm-1)!\ell_t(s)}{s^{mn-1}} \\
&=t(nm-1)!\int_{\R}\frac{(z_1\cdots z_m)^{n-m}\prod_{i<j}(z_i-z_j)^2}{(n-1)!\cdots(n-m)!\cdot (m-1)!\cdots 1!}\ddd{\nu}\,,
\end{aligned}\]
where $\R\subset \RR^{m}$ is the region given by
\begin{equation*}
\begin{aligned}
&0\ls z_1 \ls z_2 \ls \cdots \ls z_{m-1}\ls z_m \ls 1\\
&z_1+z_2+\cdots+z_{m-1}+z_{m}=t\,.
\end{aligned}
\end{equation*}
Using the fact that the integrand is symmetric under permutation of the variables, we obtain the formula in the statement.

For (ii), one proceeds as in (i) using the condition in Proposition~\ref{diagram_generic} for a standard monomial to belong to $(\overline{I_t(X)^{s+1}})^{\sat} \, / \overline{I_t(X)^{s+1}}\,\,.$

For (iii), by Theorem \ref{thm1_generic} and condition 2a) of Proposition \ref{diagram_generic} we conclude that $I_t(X)$ satisfies the assumption in Theorem \ref{thm:j=te},~(iii); the formula then follows by part (i).
\end{proof}

We have analogous statements for the symmetric and skew-symmetric cases, the proofs of which proceed along the same lines as above.

\begin{thm}\label{thm_symmetric} Let $Y$ be a generic symmetric $n \times n$ matrix of indeterminates, $t$ an integer with $0<t<n$, and set
\[
c =\frac{2^{\binom{n}{2}} \big(\binom{n+1}{2}-1\big)!}{n!(n-1)!(n-2)!\cdots 1!}\,.
\]
Then\begin{itemize}
\item[(i)] \qquad $\displaystyle
j(I_t(Y))=ct \int\limits_{\substack{[0,1]^n \\ \sum{z_i}= t}}\prod_{1\ls i<j \ls n}|z_j-z_i| \ddd{\nu}\,;$\vspace{2mm}

\item[(ii)]\qquad $\displaystyle
\e(I_t(Y))=\,\,c\binom{n+1}{2}\!\!\!\!\!\!\!\!\!\!\!\!\!\!\!\!\!\!\!\!\!\int\limits_{\substack{[0,1]^n \\  \max_i\{z_i\}+t-1\ls \sum{z_i} \ls t }}\!\!\!\!\!\!\!\!\!\!\!\!\!\!\!\!\!\!\! \prod_{1\ls i<j \ls n}|z_j-z_i| \ddd{z}\,,$\vspace{2mm}
\item[(iii)]  If $\chara(K)=0$ or $>\min\{t,n-t\}$ then,

\qquad $\displaystyle
e(A_t(Y))=c \int\limits_{\substack{[0,1]^n \\ \sum{z_i}= t}} \prod_{1\ls i<j \ls n} |z_j-z_i| \ddd{\nu}\,;$\vspace{2mm}
\end{itemize}
where $A_t(Y)$ is the algebra of minors $\F(I_t(Y))$.
\end{thm}

\begin{proof}

 The diagrams in this case have an even number of rows of each size. So if we denote by $2r_i$ the number of rows of length $i$ for $1\ls i \ls n$, the conditions in Proposition \ref{diagram_symmetric} can be written in terms of the $r_i$ as in the proof of Theorem \ref{thm_generic}.  Now, we apply Proposition \ref{W-prop} to count the number of tableaux, we compute $W_n(2r_1,\ldots,2r_n)$ whose leading term equals the one of $2^{\binom{n}{2}}W_n(r_1,\ldots,r_n)$. The rest of the proof follows as in Theorem \ref{thm_generic}, using Theorem \ref{thm1_symmetric} and Proposition \ref{diagram_symmetric} for part iii).
\end{proof}

\begin{thm}\label{thm_pfaffian} Let $Z$ be a generic skew-symmetric $n \times n$ matrix of indeterminates. Set $m=\lfloor\frac{n}{2}\rfloor$. Let $t$ an integer with $0<t<m$ and put
\[
c =\frac {\big(\binom{n}{2}-1\big)!}{m!(n-1)!(n-2)!\cdots 1!}
\]
Also, set $\delta(n)$ to be 0 if $n$ is even and 1 otherwise. Then
\begin{itemize}
\item[(i)] \qquad $\displaystyle
j(P_{2t}(Z))=ct \int\limits_{\substack{[0,1]^m \\ \sum{z_i}= t}} (z_1\cdots z_m)^{2\delta(n)} \prod_{1\ls i<j \ls m}(z_j-z_i)^4  \ddd{\nu}\,;$\vspace{2mm}

\item[(ii)] \qquad $\displaystyle
\e(P_{2t}(Z))=\,\,c\binom{n}{2}\!\!\!\!\!\!\!\!\!\!\!\!\!\!\!\!\!\!\!\!\!\!\int\limits_{\substack{[0,1]^m \\  \max_i\{z_i\}+t-1\ls \sum{z_i} \ls t }}\!\!\!\!\!\!\!\!\!\!\!\!\!\!\!\!\!\!\! (z_1\cdots z_m)^{2\delta(n)} \prod_{1\ls i<j \ls m}(z_j-z_i)^4  \ddd{z}\,,$\vspace{2mm}

\item[(iii)]  If $\chara(K)=0$ or $>\min\{2t,n-2t\}$ then,

\qquad $\displaystyle
e(A_t(Z))={c}\int\limits_{\substack{[0,1]^m \\ \sum{z_i}= t}} (z_1\cdots z_m)^{2\delta(n)} \prod_{1\ls i<j \ls m}(z_j-z_i)^4  \ddd{\nu}\,;$\vspace{2mm}
\end{itemize}
where $A_t(Z)$ is the algebra of pfaffians $\F(P_{2t}(Z))$.
\end{thm}

\begin{proof}
The proof is similar to the previous two theorems. In this case the diagrams only have rows with even size, then we compute using \linebreak ${W_n(0,r_1,0,r_2,0,\ldots,0,r_m)}$ if $n$ is even or $W_n(0,r_1,0,r_2,0,\ldots,0,r_m,0)$ if $n$ is odd. When $n=2m$ is even, in the notation of Proposition~\ref{W-prop}, write 
\[\label{Wpfaff}\tag{$\star$}
\begin{aligned}
W_{2m}&(0,r_1,0,r_2,0,\ldots,0,r_m)=\frac{\prod_{1 \ls i < j \ls 2m} \big( (B_j - B_i) - (j-i) \big)}{(n-1)! \cdots 1!}  \\
=\,&\frac{1}{(n-1)! \cdots 1!}\prod_{1 \ls i < j \ls m} \Big( \big( (B_{2j-1} - B_{2i}) - (2j-2i-1) \big) \\
& \big( (B_{2j} - B_{2i}) - (2j-2i) \big)  \big( (B_{2j-1} - B_{2i-1}) - (2j-2i) \big) \\
& \big( (B_{2j} - B_{2i-1}) - (2j-2i+1) \big)\Big) \prod_{1 \ls i \ls m} \big(B_{2i}-B_{2i-1} +1\big)\,. 
\end{aligned}
\]
Since $B_{2i}=B_{2i-1}$ for $1 \ls i \ls m$, the leading term in \eqref{Wpfaff} is
\[ \frac{\prod_{1\ls i < j \ls m} (B_{2j}-B_{2i})^4}{(n-1)! \cdots 1!}\,. \]
If $n=2m+1$ is odd, an extra factor of 
\[ \prod_{1 \ls i \ls m}  \big( (B_{2m+1} - B_{2i}) - (2m-2i+1) \big)  \big( (B_{2m+1} - B_{2i-1}) - (2j-2i+2) \big) \]
occurs in equation~\eqref{Wpfaff}, and since $B_{2m+1}=0$, the leading term in this case is
\[ \frac{(B_{2} B_{4} \cdots B_{2m})^2 \prod_{1\ls i < j \ls m} (B_{2j}-B_{2i})^4}{(n-1)! \cdots 1!}\,. \]
The rest of the proof follows as in Theorem \ref{thm_generic} using Proposition~\ref{diagram_pfaffian} and setting $z_i=B_{2i}$. For part iii), we use Theorem \ref{thm1_pfaffians} and Proposition \ref{diagram_pfaffian}.
\end{proof}
In \cite{CHST} it was shown that the $\e$-mutiplicity may be irrational. However, as the formulas above consist of the integral of a polynomial with rational coefficients over a polytope with rational vertices, we obtain the following corollary.

\begin{cor} For the ideals $I_t(X)$, $I_t(Y)$, $P_{2t}(Z)$ in Theorems~\ref{thm_generic},~\ref{thm_symmetric}, and~\ref{thm_pfaffian} above, the $\e$-multiplicity is a rational number.
\end{cor}

\begin{remark}\label{analytic_spread}
In the cases not included in Theorems \ref{thm_generic}, \ref{thm_symmetric} and \ref{thm_pfaffian} the analytic spread is not maximal, so the $j$- and $\e$-multiplicity are $0$ in these cases. More precisely:
\begin{itemize}
\item[(i)] $A_m(X)$ is the coordinate ring of the Grassmannian of $m$-dimensional subspaces in an $n$-dimensional $K$-vector space. In particular, we have that $\ell(I_m(X))=m(n-m)+1$. The formula for $e(A_m(X))$ is classical and follows by the ``postulation formula" proved by Littlewood and then by Hodge in \cite{Ho}. 
\item[(ii)] $A_n(Y)$ is a polynomial ring in one variable;
\item[(iii)] $A_m(Z)$ is a polynomial ring in one variable if $n=2m$, while it is a polynomial ring in $n$ variables if $n=2m+1$ (this follows, for example, by a stronger result of Huneke in \cite{Hun}, where he shows that these ideals are of linear type).
\end{itemize}
\end{remark}

\begin{example}\label{values}
In Table 1, we present some values of $j(I_t(X))$ and $\e(I_t(X))$ for small $t$, $m$, and $n$ using Theorem~\ref{thm_generic}. To compute the multiplicities of larger determinantal varieties via Theorems \ref{thm_generic}, \ref{thm_symmetric}, and \ref{thm_pfaffian}, there are very good specific programs to evaluate integrals of polynomial functions over rational polytopes: LattE \cite{DK} (in particular, the new version LattE integrale) or Normaliz \cite{BIS} (via the package NmzIntegrate). The algorithms used in these programs are explained, respectively, in the papers \cite{BBDKV} and \cite{BS}.
\begin{table}[ht]
\caption{Some Values of $j(I_t(X))$ and $\e(I_t(X))$}
\centering
\begin{tabular}{c c c c c}
\hline
\hline
$t$ & $m$ & $n$ & $j(I_t(X))$ & $\e(I_t(X))$ \\
[0.5ex]
\hline

2 & 3 & 3 & 2 & $1/2$ \\
2 & 3 & 4 & 64 & $341/2^4$\\
2 & 3 & 5 & 1192 &$ 62289/2^7$\\
2 & 3 & 6 & 17236 & $4195559/2^9$\\
2 & 4 & 4 & 4768 & $214865/2^{5}3$ \\
2 & 4 & 5 & 178368 & $ 1610240575/2^{6}3^{5}$\\
2 & 4 & 6 & 4888048 & $33029597513545/2^{9}3^{9}$\\
3 & 4 & 4 & 3 & $1/3$\\
3 & 4 & 5 & 2853 & $ 96631/3^5$\\
3 & 4 & 6 & 368643747 & $4134333611/3^9$\\
4 & 5 & 5 & 4 & $1/4$\\
4 & 5 & 6 & 130496 & $ 40162739/2^{12}$\\
\hline
\end{tabular}
\label{table:genericex}
\end{table}
\end{example}

\section{The Integral}\label{sec:integral}

In this section, we briefly discuss a few aspects of the integrals appearing in Section~\ref{section_main} : their meaning in random matrix theory, methods of evaluating them, and a closed formula in the case $t=m-1$ of Theorem~\ref{thm_generic}. 

To the first end, we apply a well-known transformation of Hua \cite[Section~3.3]{Hu}:
Let 
\[ \mcal{D}= \{ W \, | \, W\, \text{is $m \times m$ Hermitian, $W$ and $1-W$ are positive definite} \} \,.
\]
Then
\begin{equation*}\label{distributiontrace}
\int\limits_{\substack{[0,1]^m \\  \sum{z_i} \ls t}}(z_1\cdots z_{m})^{n-m}\prod_{i<j}(z_j-z_i)^2 \ddd{z} = \frac{(m-1)! \cdots 1!}{(2\pi)^{\binom{m}{2}}}  \!\!\!\!\!\!  \int\limits_{\substack{W\in \mcal{D} \\ \tr(W)\ls t}} \!\!\!\!\!\!\big(\det(W)\big)^{n-m} \ddd{w}
\end{equation*}
where  $\ddd{w}$ is the product of differentials over all real and imaginary components of entries of $W$. Thus, if $W$ is a random Hermitian matrix with $W$ and $1-W$ positive definite, with distribution proportional to $\big(\det(W)\big)^{n-m}$, the integral above is proportional to the probability that $\tr(W)\ls t$.

These and other related integrals have received considerable interest recently, see \cite{FW} for a survey of some of this activity. No closed forms for the integrals above are known in general. However, the case $t=1$ was settled in somewhat greater generality than the above by unpublished work of Selberg, and by Askey and Richards~\cite{AR}. For $t=1$, we have $I_1(X)$ is the homogeneous maximal ideal of the polynomial ring $K[X]$, so that
\[j(I_1(X))=\e(I_1(X))=e(A_1(X))=1\,. \]
We obtain as a corollary of Theorem~\ref{thm_generic} the following special case of the main theorem of \cite{Sel} (Selberg's Integral):
\begin{cor}
\begin{align*}
\int\limits_{\substack{z_1\dots,z_m \gs 0 \\ \sum{z_i}\ls 1}}\!\!\!\!\!\!\!\!\!(z_1\cdots z_{m})^{n-m}\prod_{i<j}(z_j-z_i)^2 \ddd{z} &= \frac{(n-1)!\cdots (n-m)!\cdot m!\cdots 1!}{(nm)!}\,.
\end{align*}
\end{cor}
Other special cases of Selberg's integral are obtained from the case $t=1$ in Theorems~\ref{thm_symmetric}, and~\ref{thm_pfaffian}.

In \cite{BBDKV}, Baldoni et al. prove a nice method to integrate a product of linear forms over a simplex. The integrands in Theorems \ref{thm_generic},~\ref{thm_symmetric}, and~\ref{thm_pfaffian} are products of linear forms, and the regions of integration in the integral formulas for the $j$-multiplicity and for the degree of the fiber cone in the theorems above are well-studied hypersimplices, with triangulations given by \cite{Sta2} or \cite{Stuu}. This provides a relatively quick method for evaluating these integrals. 

We illustrate this in more detail in the special case of $j\big(I(X)_{m-1}\big)$ --- equivalently, by \ref{thm:j=te}~(iii), of $t \cdot e\big(\F(I(X)_{m-1})\big)$ --- where the region of integration is already the simplex $\Delta\in\RR^m$ with vertices $(1,1,\ldots ,1,0,1,\ldots ,1)$. Applying \cite[Corollary~11]{BBDKV}, since ${\vol(\Delta)=\frac{1}{(m-1)!}}$, we get

\begin{align*}
\sum\limits_{M \in \NN^{\binom{m+1}{2}}} \!\! \frac{ (\sum M + m-1)!\prod_{i\ls j}t_{ij}^{M_{ij}} } {\prod_{i \ls j} (M_{ij}!)}
\int_{\Delta} {x_1^{M_{11}}\cdots x_m^{M_{mm}} \prod_{i<j} (x_i-x_j)^{M_{ij}} \ddd{\nu}}\\
= \frac{1}{\prod\limits_{k=1}^m \big( 1- \sum\limits_{h\neq k} t_{hh} + \sum\limits_{j>k} t_{kj} - \sum\limits_{i<k} t_{ik} \big)}
\end{align*}
with $\nu$ as in the proof of Theorem~\ref{thm_generic}, so that one may calculate the $j$-multiplicity by expanding the series on the right to order $m(n-1)$, retrieving the coefficient of the term with exponents $M_{ij}=2$ for $i<j$ and $M_{ij}=n-m$ for $i=j$, and multiplying by the appropriate constant, namely:
\[\displaystyle \frac{(m-1)\cdot 2^{\binom{m}{2}}}{\prod_{i=1}^{m-1}(n-i)^i\cdot m!(m-1)!\cdots 2!}\,.\]
In order to retrieve these coefficients define, for all $k=1,\ldots ,m$,
\[\ell_k=1- \sum\limits_{h\neq k} t_{hh} + \sum\limits_{j>k} t_{kj} - \sum\limits_{i<k} t_{ik}\,. \]
Denote the inverse of $\ell_k$ by
\[G_k=\sum \lambda_k((a_{ii},a_{ik})_{i\neq k})\prod_{i\neq k}t_{ii}^{a_{ii}}t_{ik}^{a_{ik}}\,.\]
It is easy to see that the coefficients of $G_k$ are:
\[\lambda_k((a_{ii},a_{ik})_{i\neq k}) = 
\displaystyle (-1)^{\sum_{i<k}a_{ik}}\frac{(\sum_{i\neq k}(a_{ii}+a_{ik}))!}{\prod_{i\neq k}(a_{ii}!a_{ik}!)}\,.
\] 
We want to compute $G=\prod_{k=1}^m G_k=\sum \lambda((a_{ij})_{i\ls j})\prod_{i\ls j}t_{ij}^{a_{ij}}$. To this goal we have to identify all $m$ terms (one for each $G_k$) whose product is $\prod_{i\ls j}t_{ij}^{a_{ij}}$. This is the set $A((a_{ij})_{i\ls j})$ consisting of elements of the form $(a_{ii}^k,a_{ik}^k)_{\substack{k=1,\ldots ,m \\ i\neq k}}$, where the $a_{ij}^k$ are natural numbers satisfying:
\begin{eqnarray*}
\sum_{k=1}^ma_{ii}^k=a_{ii} & \forall \ i=1,\ldots ,m\,,\\ 
\sum_{k=1}^ma_{pq}^k=a_{pq} & \forall \ 1\ls p<q\ls m\,,
\end{eqnarray*}
where we have put $a_{ij}^k=0$ if $i=j=k$ or $i\neq j\neq k\neq i$, and $a_{ij}^k=a_{ji}^k$ in the remaining cases. With this notation we have
\[\displaystyle  \lambda((a_{ij})_{i\ls j})=\sum\limits_{(a_{ii}^k,a_{ik}^k)_{i\neq k}\in A((a_{ij})_{i\ls  j})}\prod_{k=1}^m(-1)^{\sum_{i<k}a_{ik}^k}\frac{(\sum_{i\neq k}(a_{ii}^k+a_{ik}^k))!}{\prod_{i\neq k}(a_{ii}^k!a_{ik}^k!)} \,. \]
Recall that the coefficient that is relevant to $j(I_{m-1}(X))$ is $\lambda((a_{ij})_{i\ls j})$ with ${a_{ii}=n-m}$ and $a_{ij}=2$ for $i<j$. For such $(a_{ij})_{ij}$ the set  $A((a_{ij})_{i\ls j})$ can be identified with the set $A(m,n)$ of all pairs of $m\times m$ matrices $(a_{ij},b_{ij})_{ij}$ with natural entries such that
\begin{eqnarray*}
a_{ij}=b_{ij}=0 & \mbox{if } i=j \\
\sum_{j=1}^ma_{ij}=n-m & \forall \ i=1,\ldots ,m\,,\\ 
b_{pq}+b_{qp}=2 & \forall \ 1\ls p<q \ls m\,.
\end{eqnarray*}
Finally we get the following:
\begin{prop}
The $j$-multiplicity of $I_{m-1}(X)$ is
{\small \begin{equation*}
\frac{(m-1)\cdot 2^{\binom{m}{2}}}{\prod_{i=1}^{m-1}(n-i)^i\cdot m!(m-1)!\cdots 1!} \left( 
\nonumber \sum\limits_{(a_{ij},b_{ij})_{ij}\in A(m,n)}\prod_{j=1}^m(-1)^{\sum_{i<j}b_{ij}}\frac{(\sum_{i=1}^m(a_{ij}+b_{ij}))!}{\prod_{i=1}^m(a_{ij}!b_{ij}!)}\right)\!\,.
\end{equation*}}
\end{prop}

\section*{Acknowledgements}

The authors thank the Mathematical Sciences Research Institute, Berkeley, CA, where the discussions originating this work started. Also, thanks to Aldo Conca for enlightening discussions about Subsection \ref{subrat}, to Bernd Ulrich for guiding us to the proof of the second point of Theorem \ref{thm:j=te}, and to Donald Richards for explaining the connections with random matrix theory. The authors also wish to thank Winfried Bruns for his help in computing the values in the table of Example \ref{values} using the software NmzIntegrate. The first two authors also thank their Ph.D. advisors, Anurag Singh and Bernd Ulrich, respectively.
The authors are also grateful to the referee for her or his helpful corrections.

\end{document}